\documentclass[12pt]{article}
\usepackage{amsmath, amsfonts, amsthm, amssymb, verbatim}
\usepackage{graphicx}
\usepackage[mathcal]{euscript}
\usepackage{amstext}
\usepackage{amssymb,mathrsfs}
\usepackage[latin1]{inputenc}
\newcommand \be     {\begin{equation}}
\newcommand \ee     {\end{equation}}
\newcommand \RR      {\mathbb{R}}
\newcommand \NN      {\mathbb{N}}
\newcommand \del     \partial

\newcommand{\clg}[1]{{\mathcal{#1}}}

\newcommand{\td}[1]{{\tilde{#1}}}

\def\XXint#1#2#3{{\setbox0=\hbox{$#1{#2#3}{\int}$}
\vcenter{\hbox{$#2#3$}}\kern-.5\wd0}}

 \newcommand{\R}{\mathbb R}
 
 \newcommand{\Z}{\mathbb Z}
 
 \newcommand{\C}{\mathbb C}

 \newcommand{\p}{\partial}

\newtheorem{theorem}{Theorem}[section]

 \newtheorem{remark}[theorem]{Remark}
\newtheorem{lemma}[theorem]{Lemma}

\newtheorem{definition}[theorem]{Definition}
\newtheorem{example}[theorem]{Example}

 \def\beqs{\begin{eqnarray*}}
 \def\enqs{\end{eqnarray*}}
 \def\beq{\begin{eqnarray}}
 \def\enq{\end{eqnarray}}

\begin{document}
\title{Persistence
property in weighted Sobolev spaces for nonlinear dispersive
equations }

\author{X. Carvajal$^1$, W. Neves$^1$}

\date{}

\maketitle

\footnotetext[1]{ Instituto de Matem\'atica, Universidade Federal
do Rio de Janeiro, C.P. 68530, Cidade Universit\'aria 21945-970,
Rio de Janeiro, Brazil. E-mail: {\sl carvajal@im.ufrj.br,
wladimir@im.ufrj.br.}

}

%

%
\begin{abstract}
We generalize the Abstract Interpolation Lemma proved by the
authors in \cite{CN1}. Using this extension, we show in a more
general context, the persistence property for the generalized
Korteweg-de Vries equation, see \eqref{IVP}, in the weighted
Sobolev space with low regularity in the weight. The method used
can be applied for other nonlinear dispersive models, for instance
the multidimensional nonlinear Schrödinger equation.
\end{abstract}
%

\maketitle

\section{Introduction} \label{IN}

We are mainly concerned with the question of the persistence
property in weighted Sobolev spaces for dispersive partial
differential equations. Thus, the aim of this study is to
generalize the Abstract Interpolation Lemma proved by the authors
in \cite{CN1}, and to apply this new result to show, in a more
general context, the persistence property of the initial-value
problem for nonlinear dispersive equations. To be more precise,
let us recall the persistence result we established in \cite{CN1}
for the Cauchy Problem for higher order nonlinear Schrödinger
equation, that is
\begin{equation}\label{IVPHOSEQ}
\begin{cases}
\p_t u + i \, a\,\p^2_xu + b\,\p^3_x u + i \,c \,|u|^{2}u + d\,
|u|^2 \p_x u+e\, u^2\p_x\bar{u}=  0, \quad (t,x) \in \R^2,\\
u(x,0)  =  u_0(x),
\end{cases}
\end{equation}
where $u$ is a complex valued function, $a, b, c, d$ and $e$ are
real parameters and $u_0$ is a given initial data. And, the main
theorem on that paper:
\begin{theorem}\label{t1.2}
The IVP (\ref{IVPHOSEQ}) is globally well-posed in
$\mathcal{X}^{2,\theta}$ for any  $0\le \theta \le 1$ fixed.
Moreover, the solution $u$ of (\ref{IVP}) satisfies, for each $t
\in [-T,T]$
\begin{equation*}
  \|u(t)\|_{L^2(d\dot{\mu}_\theta)}^2  \leq C \; \Big(\|u_0\|_{L^2}^2
  + \|u_0\|_{L^2(d\dot{\mu}_\theta)}^2 \, + 1 \Big),
\end{equation*}
where $C=  C
(\theta,\|u(t)\|_{H^s},\|u(0)\|_{L^2},\|u_x(0)\|_{L^2},\|u_{xx}(0)\|_{L^2},T),
s> 1/2$.
\end{theorem}
The notion of well-posedness for dispersive equations is given
below, and the particular notations used throughout this paper are
given in Section \ref{Notation}. Therefore, one of the main issues
of this article is to extend the persistence property proved
before for $\theta \in [0,1]$ to more general values of the
exponent $\theta$. In particular, we explore our strategy on the
generalized KdV equation, see \eqref{IVP} below, i.e., we consider
the 1-dimensional case. However, the extension of the Abstract
Interpolation Lemma proved in this paper to show the persistence
property for more general exponents $\theta$, also allows us to
demonstrate the persistence property for multi-dimensional
equations as presented in this paper.

\bigskip
Consider the initial value problem (IVP)
\begin{equation}\label{IVP}
\begin{cases}
\p_t u +  a(u) \partial_xu + {\partial_x}^{\!\!\!3} u= 0, \quad (t,x) \in \R^2,\\
u(0,x)  =  u_0(x),
\end{cases}
\end{equation}
where $u$ is the real valued function we are seeking, $u_0$ is the
initial-data given in some convenient space, and $a(u)$ is a given
$C^\infty$ (weaker differentiability is sufficient for most
results) real value function. Moreover, we may assume that $a(u)$
satisfies, as in Kato \cite{KT1}, the following condition
\begin{equation}
\label{AC}
  \begin{aligned}
  \limsup_{|\lambda| \to \infty} \frac{2}{|\lambda|^6}
  \int_0^\lambda (\lambda - s) \; a(s) \; ds \leq 0.
  \end{aligned}
\end{equation}

Now, we introduce the typical notion of well-posedness that we are
going to use throughout this paper. First, we consider the
integral equation associated with \eqref{IVP}
\begin{equation}
\label{integreq}
   u(t)= U(t)\, u_0 + \int_0^t U(t-\tau) \;
   a(u(\tau)) \; \partial_x u(\tau) \, d\tau,
\end{equation}
where $U(t)$ is the unitary group, solution of the linear KdV
equation. It is not difficult to show that, if $u$ is a solution
for the Cauchy Problem \eqref{IVP}, then it satisfies
\eqref{integreq}. Then, we have the following
\begin{definition}
\label{WELLPD} Let $X$, $Y$ be two Banach spaces, such that $X$ is
continuously embedded in $Y$. Suppose that, for each $u_0 \in X$,
there exists $T>0$, and a unique function
\begin{equation}
\label{UNIQSOL}
  u \in C([0,T]; X)
\end{equation}
satisfying \eqref{integreq} for all $t \in [0,T]$, and also
$\partial_t u \in C((0,T]; Y)$. The Cauchy Problem \eqref{IVP} is
said to be locally well-posed in $X$, when the map $u_0 \mapsto u$
is continuous from $X$ to $C([0,T]; X)$. If $T$ can be taken
arbitrary large, then \eqref{IVP} is said globally rather than
locally well-posed in $X$. Moreover, \eqref{UNIQSOL} implies the
persistence property of the initial data.
\end{definition}

\medskip
If we consider the initial data in Sobolev spaces with sufficient
regularity, for example in $H^s(\RR)$, $s\geq 2$, it is not
difficult to prove the unique existence of the solution of the IVP
\eqref{IVP} in the weighted Sobolev spaces. However, proving the
persistence property, also continuous dependence, is not so easy
and it is quite involved when we are working in weighted Sobolev
spaces. Our main focus in this paper is to show the persistence
property, with respect to more general exponents, as explained
below. To accomplish this, in the present paper we establish an
extension of the abstract interpolation lemma proved in
\cite{CN1}.

\medskip In fact, the interpolation extension proved here is quite
general and applies to several dispersive equations provided they
satisfy certain {\em a priori} estimates. These {\em a priori}
estimates are related to the conserved quantities and are as
follows:

\begin{equation}\label{ap-1}
\|u(t)\|_{L^2} \le C \|u_0\|_{L^2},
\end{equation}
%
%
\begin{equation}\label{ap-3}
\|u(t)\|_{\dot{H}^{a(r)}} \le A_1(\|u_0\|_{{H}^{{a(r)}}}),
\end{equation}
and
\begin{equation}\label{ap-4}
\|u(t)\|_{L^2(d\dot{\mu}_r)} \le C \|u_0\|_{L^2(d\dot{\mu}_r)}+
A_2(\|u_0\|_{{H}^{{a(r)}}}),
\end{equation}
where $a(r) \ge 1, r \in \mathbb{Z}^+$, $A_j$ are nonnegative
continuous functions with $A_1(0)=0$, $A_2(0)=0$. Here, we
consider that the IVP \eqref{IVP} satisfies
\eqref{ap-1}--\eqref{ap-4}, for that we address the reader to Kato
\cite{KT1} as we are going to precise below. A typical equation
that satisfies the properties \eqref{ap-1}--\eqref{ap-4} listed
above is the IVP associated to the generalized Korteweg-de Vries
(gKdV) equation,
\begin{equation}\label{IVP-1}
\begin{cases}
\p_t u + u^k\p_xu + \p_{x}^3u= 0, \quad (t,x) \in \R^2, k = 1, 2, 3, \cdots \\
u(x,0)  =  u_0(x).
\end{cases}
\end{equation}

\medskip
Before stating the main result of this work, we discuss some
similar results, previously obtained in the same direction of the
main issue of this paper. The IVP associated to the Nonlinear
Schr\"odinger (NLS) equation
\begin{equation}\label{nls-1}
\begin{cases}
 i\partial_tu +\Delta u = \mu |u|^{\alpha-1}u, \qquad \mu = \pm 1, \; \alpha >1,\, x \in \R^n, \, t\in \R\\
 u(x,0) = u_0(x),
 \end{cases}
 \end{equation}
 has been studied in \cite{HNT} for given
 data in the weighted Sobolev spaces.
 More precisely, the following theorem that
 deals with the persistence property has been proved in
 \cite{HNT}:
 \begin{theorem}\label{th.1}
 Suppose that $u_0\in H^s(\R^n)\cap L^2(|x|^{2m}dx)$, $m\in \Z^+$, with $m\leq \alpha-1$ if $\alpha$ is not an odd integer.

\noindent
A. If $s\geq m$, then there exist $T=T(\|u_0\|_{s,2})>0$ and a unique solution $u=u(x,t)$ of the IVP \eqref{nls-1} with
\begin{equation}\label{eq-m1}
u\in C([-T, T]; H^s\cap L^2(|x|^{2m}dx))\cap L^q([-T, T];L_s^p\cap L^p(|x|^{2m}dx)).
\end{equation}

\noindent
B. If $1\leq s <m$, then \eqref{eq-m1} holds with $[s]$ instead of $m$, and
\begin{equation}\label{eq-m2}
\Gamma^{\beta}u=(x_j+2it\partial_{x_j})^{\beta}u\in C([-T, T]; L^2)\cap L^q([-T, T]; L^p),
\end{equation}
for any $\beta \in (\Z^+)^n$ with $|\beta|\leq m$.
\end{theorem}

 The power $m$ of the weight in Theorem \ref{th.1} is
 assumed to be a positive integer. In the recent
 study of Nahas and Ponce \cite{NP}, this restriction in $m$ is
 relaxed by proving that the persistence property holds for positive
real $m$. To be more precise, the result in \cite{NP} is the
following
 \begin{theorem}\label{th.2}
 Suppose that $u_0\in H^s(\R^n)\cap L^2(|x|^{2m}dx)$, $m>0$, with $m\leq \alpha-1$ if $\alpha$ is not an odd integer.

\noindent
A. If $s\geq m$, then there exist $T=T(\|u_0\|_{s,2})>0$ and a unique solution $u=u(x,t)$ of the IVP \eqref{nls-1} with
\begin{equation}\label{eq-m3}
u\in C([-T, T]; H^s\cap L^2(|x|^{2m}dx))\cap L^q([-T, T];L_s^p\cap L^p(|x|^{2m}dx)).
\end{equation}

\noindent
B. If $1\leq s <m$, then \eqref{eq-m3} holds with $[s]$ instead of $m$, and
\begin{equation}\label{eq-m4}
\Gamma^b\Gamma^{\beta}u\in C([-T, T]; L^2)\cap L^q([-T, T]; L^p),
\end{equation}
where $\Gamma^b = e^{i|x|^2/4t}2^bt^bD^b(e^{i|x|^2/4t}.)$ with $|\beta|=[m]$ and $b=m-[m]$.
 \end{theorem}

In the next section, see the IVP \eqref{IVPgeral}, Theorem
\ref{teorlinearpers} and Remark \ref{RCF}, we establish the
conditions to apply our technique, and hence we obtain similar
results for the above NLS equation.

\medskip
Now, we recall that Kato \cite{KT1} studied the IVP \eqref{IVP}
for a given initial data in the weighted Sobolev spaces and proved
the following result.
\begin{theorem}\label{kato}
Let $r$ be a positive integer. Then, the IVP \eqref{IVP} is
locally well-posed in weighted Sobolev spaces
$\mathcal{X}^{2r,r}$, and globally well-posed in
$\mathcal{X}^{2r,r}$ if the initial data satisfies $\|u_0\|_{L^2}<
\gamma$, for some positive $\gamma$.
\end{theorem}

The proof of Theorem \ref{kato} is given in Kato's Theorem 8.1 and
Theorem 8.2, see\cite{KT1}.

\bigskip
In fact, the persistence property for dispersive
equations has been discussed recently, as in Nahas
\cite{N8}, and Nahas, Ponce \cite{NP9}. Moreover, the results
on \cite{N8} were extended recently by Nahas to generalized KdV
equation, see \cite{N10}, also we address the first work of the authors 
in this direction, see \cite{CN1}.
In this paper we are interested in removing the requirement that
the power of the weight in Theorem \ref{kato} is an integer, by
proving that a similar result is obtained for non integer values
of $r$.  One of the main results of this article is the following
\begin{theorem}\label{teomain}
Assume $r \geq 1$. If the IVP \eqref{IVP} is local well-posed in
$H^s$ for $s \geq 2 r$ and satisfies the {\em a priori} estimates
\eqref{ap-1}--\eqref{ap-4}, then the IVP \eqref{IVP} has the
properties of the unique existence and persistence in weighted
Sobolev spaces $\mathcal{X}^{s,\theta}$, for $s \geq 2 r$ and
$\theta \in [0,r]$.
\end{theorem}

One observes that, in the above theorem $r$ is a real number.
Moreover, from the proof of the Theorem \ref{teomain}, it can be
inferred that, if one has local well-posedness result for given
data in $H^s$ and if the model under consideration satisfies  {\em
a priori} estimates \eqref{ap-1}-\eqref{ap-4}, then with the help
of an abstract interpolation lemma, it is easy to prove
persistence property in weighted Sobolev spaces.

As an application of Theorem \ref{teomain} we have the following
result

\begin{theorem}\label{teomain1}
Let $r \geq 1$ be a real number. Then, the IVP for the gKdV
equation \eqref{IVP-1} is locally well-posed in weighted Sobolev
spaces $\mathcal{X}^{s,\theta}$, for $s \ge 2r$ and $0 \le \theta
\le r$. Moreover, globally well-posed in $\mathcal{X}^{s,\theta}$,
for $0 \le \theta \le r$ and $s \ge 2r$ $($for $k \geq 4$ the initial data
must satisfies $\|u_0\|_{L^2}< \gamma$, for some $\gamma>0)$.
\end{theorem}

The paper is organized as follows: In the rest of this section we
fix the notation and some background used throughout the paper.
The Abstract Interpolation Lemma is given at Section \ref{CL}. In
Section 3, we first show some conserved quantities, and prove a
nonlinear estimate. Then, we formulate the approximate problems
associated to the IVP \eqref{IVP} from them, we gain continuous
dependence in $H^s$ norms,
which is used to show mainly Theorem \ref{teomain1} at the end of this section.\\

\subsection{Notation and background} \label{Notation}

We follow the notations introduced in earlier paper \cite{CN1}.
For the sake of clarity we recall most of them here, clearly
adapted for the multidimensional setting and for the more general
case of $\theta \in [0,r]$, $r \geq 1$. Moreover, we present some
results used through the paper.

We use $dx$ to denote the Lebesgue measure on $\R^n$ and,
$$
  \begin{aligned}
  d\mu_\theta(x)&:= (1 + \|x\|^2)^ \theta \; dx,
  \\
    d\dot{\mu}_\theta(x)&:= \|x\|^{2 \theta} \; dx
  \end{aligned}
$$
to denote the Lebesgue-Stieltjes measures on $\RR^n$. Hence, given
a set $X$, a measurable function $f \in L^2(X;d\mu_\theta)$ means
that
$$
  \|f\|_{L^2(X ; d\mu_\theta)}^2= \int_X |f(x)|^2 \; d\mu_\theta(x)< \infty.
$$
When $X= \RR^n$, we write: $L^2(d\mu_\theta) \equiv
L^2(\RR^n;d\mu_\theta)$, and for simplicity
$$
  L^2 \equiv L^2(d\mu_0), \quad L^2(d\mu)
  \equiv L^2(d\mu_1)
$$
and similarly for the measure $d\dot{\mu}_\theta$. We will use the
Lebesgue space-time $L_{x}^{p}\mathcal{L}_{\tau}^{q}$ endowed with
the norm
$$
\|f\|_{L_{x}^{p}\mathcal{L}_{\tau}^{q}} = \big\|
\|f\|_{\mathcal{L}_{\tau}^{q}} \big\|_{L_{x}^{p}} = \Big(
\int_{\R} \Big( \int _{0}^{\tau} |f(x,t)|^{q} dt \Big)^{p/q} dx
\Big)^{1/p} \quad (1 \leq p,q < \infty).
$$
When the integration in the time variable is on the whole real
line, we use the notation $\|f\|_{L_{x}^{p}L_{t}^{q}}$. The
notation $\|u\|_{L^p}$ is used when there is no doubt about the
variable of integration. We adopt similar notations as above, when
$p$ or $q$ are $\infty$.
As usual, $H^s \equiv H^s(\RR^n)$, $\dot{H}^s \equiv
\dot{H}^s(\RR^n)$ are the classic Sobolev spaces in $\RR^n$,
endowed respectively with the norms
$$
\|f\|_{H^s}:= \|\widehat{f}\|_{L^2(d\mu_s)},
\quad
\|f\|_{\dot{H}^s}:=\|\widehat{f}\|_{L^2(d\dot{\mu}_s)}.
$$

\medskip
We study in this work the solutions of dispersive equations in the
weighted Sobolev spaces $\mathcal{X}^{s,\theta}$, defined as
\begin{equation}
\label{SSWW}
  \mathcal{X}^{s,\theta}:=H^s \cap L^2(d\mu_\theta),
\end{equation}
with the norm
$$
\| f\|_{\mathcal{X}^{s,\theta}}:=\| f\|_{H^{s}}+\|
f\|_{L^2(d\mu_\theta)}.
$$
We remark that, $\mathcal{X}^{s,r} \subseteq
\mathcal{X}^{s,\theta}$, for all $s\in \mathbb{R}$ and  $\theta\in
[0,r]$. Indeed, using H\"older's inequality, we have
\begin{equation}
\label{interpolx}
  \|f\|_{L^2(d\dot{\mu}_\theta)} \leq \|f\|_{L^2}^{1-\theta/r} \;
  \|f\|_{L^2(d\dot{\mu}_r)}^{\theta/r}.
\end{equation}
Moreover, we recall the classical notation of pseudo-differential
operators. For any real number $m$, we define the set
$$
  \mathcal{S}^m:= \{a \in C^\infty(\mathbb{R}^{2n};\C):
  |\partial_{x}^{\alpha}\partial_{\xi}^{\beta} a(x,\xi)| \leq
  C_{\alpha,\beta}(1+|\xi|)^{m-|\beta|}, \; \forall \alpha,\beta \in
  (\mathbb{Z}^+)^n\}.
$$
For $a \in \mathcal{S}^m$, we consider the differential operator
$a(x,D)$, defined for any $f \in \mathbb{S}(\R^n)$ in the
following sense
$$
  \widehat{\left(a(x,D)f\right)}(\xi)=a(x,\xi)\widehat{f}(\xi).
$$
The proof of the next two lemmas can be found in \cite{NP}.
\begin{lemma}\label{oper}
If $a\in\mathcal{S}^{0}$, then for each $b > 0$
$$
  a(x,D): L^{2}(\mathbb{R}^{n};d\mu_{b}) \to
  L^{2}(\mathbb{R}^{n};d\mu_{b})
$$
is a bounded differential operator.
\end{lemma}
\begin{lemma}\label{opera-l}
Let $a,b>0$. If $D^a f \in L^2(\mathbb{R}^n)$ and
$f \in L^2(\R^n; d\mu_b)$, then for each $\theta \in [0,1]$
\begin{equation}
   \|D^{(1-\theta)a} f\|_{L^2(d\mu_{\theta b})} \leq C
   \; \|f\|_{L^2(d\mu_b)}^{\theta} \; \|D^{a}
   f\|_{L^2}^{1-\theta}.
\end{equation}
\end{lemma}

Now, applying  Lemma \ref{oper} we have the following
\begin{lemma}
\label{multindex} Let $\beta  \in (\Z^+)^n$ be a multi-index and
$b>0$ fixed. If $f \in \mathbb{S}(\R^n)$, then
\begin{equation}
  \| \partial^{\beta} f \|_{{L^2(d\dot{\mu}_b)}} \le C \|
  D^{|\beta|}f\|_{{L^2(d\dot{\mu}_b)}} +C \|f\|_{{L^2(d\dot{\mu}_b)}}.
\end{equation}

\end{lemma}
\begin{proof}
Let us consider $a(x,\xi)=
\dfrac{\xi^{\beta}}{(1+|\xi|^2)^{|\beta|/2}}$, we can see that, $
a \in\mathcal{S}^{0}$. Then, applying Lemma \ref{oper}, the
associated operator a(x,D) is bounded in
$L^{2}(\mathbb{R}^{n};d\mu_{b})$. Therefore, it follows that
\begin{equation}
\label{doslema}
   \|a(\cdot,D)g\|_{L^2(d\dot{\mu}_b)} \le C \|g\|_{L^2(d\dot{\mu}_b)}.
\end{equation}
If $\widehat{J^\beta f} (\xi)=
   (1+|\xi|^2)^{|\beta|/2} \, \widehat{f}(\xi)$, considering $g=J^{\beta}f$, then
   $a(D)g=(1/i^{|\beta|})\partial^{\beta}f$ and the lemma is proved.
\end{proof}

Now, we consider the following evolution equation
\begin{equation}
\label{IVPgeral}
\begin{cases}
\p_t u +Lu +
F(u, \nabla_x u)=  0, \quad (t,x) \in \R \times\R^n,\\
u(0,x)  =  u_0(x),
\end{cases}
\end{equation}
where the linear part of the equation $Lu$, is defined by
$$
  \widehat{Lu}(\xi)= i \, h(\xi) \, \widehat{u}(\xi),
$$
for some polynomial symbol $h(\xi)$ real valued, and $F(x,y)$ is a
function with $F(0,0)=0$ (for the KdV equation $h(\xi)= -\xi^3$,
$\xi \in \R$, $F(x,y)= a(x) y$, and for the non-linear
Schr\"odinger equation $h(\xi)=\sum_{k=1}^{n}\xi^{2e_k}= |\xi|^2$
where $e_k$ is the k-th unit vector, $\xi \in \R^n$, $F(x,y)=
|x|^{\alpha-1} x$, $\alpha>1$).
\begin{theorem}\label{teorlinearpers}
Let $r \ge 1$ and $u\in C([-T,T]; \mathcal{X}^{s,r})$ be a smooth
solution of the linear IVP
\begin{equation}\label{IVPgerallinear}
\begin{cases}
\p_t u +Lu
=  0, \quad (t,x) \in \R \times\R^n,\\
u(0,x)  =  u_0(x),
\end{cases}
\end{equation}
where the linear operator $L$ is defined with symbol
$h(\xi)=\sum_{j=1}^{p}C_j \xi^{\beta_j}$, $\xi \in \R^n$, $\beta_j
\in (\Z^+)^n$, $|\beta_j|>1$, $j=1,\dots,p$. Then, $u$ satisfies
the inequality \eqref{ap-4} with \be\label{adere}
a(r)=(\max_{j=1,\dots,p}|\beta_j|-1) \; r. \ee
\end{theorem}
\begin{proof}
%
By the Bona-Smith approximation argument, we can suppose that
$u(t) \in S(\R^n)$ or in some $\mathcal{X}^{s_0,r}$ with $s \ll
s_0 $. Moreover, without lost of generality, we can suppose that
$h(\xi)=\xi^{\beta}$, for some multi-index $\beta$, $|\beta|>1$.
Multiplying \eqref{IVPgerallinear} by $|x|^{2r} \overline{u}$,
taking the real part and integrating, we have
\begin{align}\label{persxaver}
0=\partial_t \int |x|^{2r}|u|^2 dx+ 2 \textrm{Re}\, \int (x
\cdot\overline{x})^r \overline{u} \,Lu dx.
\end{align}
Using the notation of multi-indices $\alpha=(\alpha_1, \dots,
\alpha_n)$, $\alpha_j \in \Z^+$, $j=1,\dots,n$, we have
respectively the multinomial and Leibniz formula
\be\label{multin} \left(\sum_{j=1}^{n} x_j^2\right)^r=
\sum_{|\alpha|=r}\binom{r}{\alpha}x^{2\alpha}, \quad
\partial^{\alpha} (f(\xi)g(\xi))=\sum_{\eta \le
\alpha}\binom{\alpha}{\eta}\partial^{\eta}
f(\xi)\partial^{\alpha-\eta} g(\xi). \ee
Applying the definition of the Fourier transform, we obtain
\be\label{persistencia1}\partial_{\xi}^{\alpha}
\widehat{u}(\xi)=(-i)^{|\alpha|}\widehat{x^{\alpha}  u}  (\xi) \ee
and by the multinomial formula, Plancherel equality and
\eqref{persistencia1}, we can write
\be
\label{persistencia3}
  \int |x|^{2r} |u(x)|^2 dx= (-1)^{|\alpha|} \int
  \sum_{|\alpha|=r}\binom{r}{\alpha} \left|\partial_{\xi}^{\alpha}
  \widehat{u}(\xi) \right|^2 d\xi.
\ee
Now, considering the second term in \eqref{persxaver}, we have
\begin{align}\label{persxaver7}
\int (x \cdot\overline{x})^r \overline{u} \,Lu dx&= \int  \left(\sum_{j=1}^{n} x_j^2\right)^r\,\overline{u} \,Lu\,dx \nonumber\\
&= \int  \sum_{|\alpha|=r}\binom{r}{\alpha}x^{2\alpha}\,\overline{u} \,Lu\,dx\nonumber\\
&= \sum_{|\alpha|=r}\binom{r}{\alpha}\int  x^{\alpha}\,\overline{u} \,x^{\alpha}\,Lu\,dx,\nonumber\\
&= \sum_{|\alpha|=r}\binom{r}{\alpha}\int
\overline{\widehat{x^{\alpha}\,u}
}\,\widehat{x^{\alpha}\,Lu}\,d\xi,
\end{align}
where in the last inequality we used Plancherel equality. By
Leibniz formula, identity \eqref{persistencia1} and definition of
$L$ with $h(\xi)=\xi^{\beta}$,  we have
\begin{align}\label{persxaver8}
\int \overline{\widehat{x^{\alpha}\,u} }\,\widehat{x^{\alpha}\,Lu}\,d\xi &= (-1)^{|\alpha|}
\int  \overline{\partial_{\xi}^{\alpha} \widehat{u}(\xi)}\,\partial_{\xi}^{\alpha} \widehat{Lu}(\xi)d\xi\nonumber\\
&= i \; (-1)^{|\alpha|} \int  \overline{\partial_{\xi}^{\alpha}
\widehat{u}(\xi)}\,\partial_{\xi}^{\alpha}
(\,h(\xi)\widehat{u}(\xi)\,)d\xi\nonumber
\\
&= i \; (-1)^{|\alpha|} \int  \overline{\partial_{\xi}^{\alpha}
\widehat{u}(\xi)}\,\sum_{\eta \le
\alpha}\binom{\alpha}{\eta}(\partial^{\eta}
\xi^{\beta})\,(\partial_{\xi}^{\alpha-\eta}\widehat{u}(\xi)\,)
d\xi.
\end{align}
One observes that, when $\eta=(0,\dots,0):={\bf 0}$ in
\eqref{persxaver8}, we obtain
$$i \; (-1)^{|\alpha|} \int
|\partial_{\xi}^{\alpha} \widehat{u}(\xi)|^2 d\xi
$$
and thus this term in \eqref{persxaver} is equal to zero. We
conclude from \eqref{persistencia3}, \eqref{persxaver7} and
\eqref{persxaver8} that
\begin{align}\label{perscxaver12}
2\textrm{Re}\int &|x|^{2r}\, \overline{u}\, Lu\, dx \nonumber\\
&= 2 \textrm{Re}\,i \; (-1)^{|\alpha|} \sum_{|\alpha|=r}\binom{r}{\alpha} \int  \overline{\partial_{\xi}^{\alpha} \widehat{u}(\xi)}\,\sum_{\stackrel{\eta \le \alpha}{\eta \neq {\bf 0}}}\binom{\alpha}{\eta}(\partial^{\eta} \xi^{\beta})\,(\partial_{\xi}^{\alpha-\eta}\widehat{u}(\xi)\,) d\xi\nonumber\\
& \le \sum_{|\alpha|=r}\binom{r}{\alpha} \int  \left(|\partial_{\xi}^{\alpha} \widehat{u}(\xi)|^2 +C_r \sum_{\stackrel{\eta \le \alpha}{\eta \neq {\bf 0}}}\binom{\alpha}{\eta}^2\left|\partial^{\eta} \xi^{\beta}\,\partial_{\xi}^{\alpha-\eta}\widehat{u}(\xi)\,\right|^2 \right) d\xi\nonumber\\
& \le \int |x|^{2r} |u(x)|^2
dx+C_r\sum_{|\alpha|=r}\binom{r}{\alpha} \sum_{\stackrel{\eta \le
\alpha}{\eta \neq {\bf 0}}}\binom{\alpha}{\eta}^2\int
\left|\partial^{\eta}
\xi^{\beta}\,\partial_{\xi}^{\alpha-\eta}\widehat{u}(\xi)\,\right|^2
d\xi.
\end{align}
In order to estimate the second term in \eqref{perscxaver12}, we
consider a multi-index $\eta \le \alpha$, $\eta \neq {\bf 0}$, and
the expression $J(\alpha,\beta,\eta)= \|\partial^{\eta}
\xi^{\beta}\,\partial_{\xi}^{\alpha-\eta}\widehat{u}(\xi)\|_{L_\xi^2}$.
Then, for ${\bf 0} \neq \eta \le  \alpha$, $\eta \le \beta$,
using \eqref{persistencia1},
Plancherel equality, and Leibniz formula, it follows that
\begin{align}
J(\alpha,\beta,\eta)&=\|(\partial^{\eta} \xi^{\beta})\,\partial_{\xi}^{\alpha-\eta}\widehat{u}(\xi)\|_{L_{\xi}^2}=\dfrac{\beta!}{(\beta-\eta)!}\| \xi^{\beta-\eta}\,\partial_{\xi}^{\alpha-\eta}\widehat{u}(\xi)\|_{L_{\xi}^2}\nonumber\\
&=\dfrac{\beta!}{(\beta-\eta)!}\| \xi^{\beta-\eta}\,\widehat{(x^{\alpha-\eta}u)}(\xi)\|_{L_{\xi}^2}\nonumber\\
&=\dfrac{\beta!}{(\beta-\eta)!}\| \partial_x^{\beta-\eta}\,(x^{\alpha-\eta}u)\|_{L_x^2}\nonumber\\
& \le\dfrac{\beta!}{(\beta-\eta)!}\sum_{\nu\le\beta-\eta
}\binom{\beta-\eta}{\nu}\|
(\partial^{\nu}x^{\alpha-\eta})\,(\partial_x^{\beta-\eta-\nu}u)\|_{L_x^2}.
\end{align}
Now, we proceed to estimate $\|
(\partial^{\nu}x^{\alpha-\eta})\,(\partial_x^{\beta-\eta-\nu}u)\|_{L_x^2}$.
We know that, the function $\partial^{\nu}x^{\alpha-\eta}\neq 0$
if $\nu \le \alpha -\eta$ and zero otherwise. Thus we suppose that
$\nu \le \alpha-\eta$, $\nu \le  \beta -\eta$, and since $\eta
\neq {\bf 0}$, we have
$$r_0=|\alpha -\eta-\nu|=|\alpha|-|\eta|-|\nu| =r-|\eta|-|\nu|<r, $$
and
 $$r_1=|\beta -\eta-\nu|=|\beta|-|\eta|-|\nu| <|\beta|. $$
Therefore, applying Lemma \ref{multindex} we obtain
 \begin{align}\label{finalc}
\| (\partial^{\nu}x^{\alpha-\eta})\,(\partial_x^{\beta-\eta-\nu}u)\|_{L_x^2}&=\dfrac{(\alpha-\eta)!}{(\alpha-\eta-\nu)!}\| x^{\alpha-\eta-\nu}\,\partial_x^{\beta-\eta-\nu}u\|_{L_x^2}\nonumber\\
&\le \dfrac{(\alpha-\eta)!}{(\alpha-\eta-\nu)!}\| |x|^{r_0}\,\partial_x^{\beta-\eta-\nu}u\|_{L_x^2}\nonumber\\
&\le
C_{\alpha,\eta,\nu}\left( \| |x|^{r_0}\,D_x^{r_1}u\|_{L_x^2}+\|
|x|^{r_0}u\|_{L_x^2}\right).
\end{align}
We observe that, $\eta \neq {\bf 0}$ implies $|\eta|+|\nu| \ge 1$
and this inequality implies
$$
1-\dfrac{|\eta|+|\nu|}{r} \le
1-\dfrac{|\beta|-|\eta|-|\nu|}{(|\beta|-1)r}.
$$
Now we choose $\theta$ such that $ 1-\dfrac{|\eta|+|\nu|}{r} \le
\theta \le 1-\dfrac{|\beta|-|\eta|-|\nu|}{(|\beta|-1)r} $, it
follows that $\theta \in [0,1] $. Thus applying the Intermediate
Value Theorem, there exists $ b \in [0, r]$ and there exists $ a
\in [0, (|\beta|-1)r]$ such that $r_0= \theta b$ and
$r_1=(1-\theta)a$. Using Lemma \ref{opera-l} and the interpolation
\eqref{interpolx}, we obtain
\begin{align*}
\|
(\partial^{\nu}x^{\alpha-\eta})\,(\partial_x^{\beta-\eta-\nu}u)\|_{L_x^2}
\le C \| u\|_{L^2(d\mu_r)}^{\theta}
\|u\|_{H^{(|\beta|-1)r}}^{1-\theta}+C \|x^r u\|_{L_x^2}^{(1-
\kappa_0)}\|u\|_{L^2}^{\kappa_0},
\end{align*}
where $\kappa_0=r_0/r$, and this concludes the proof of the
theorem.
\end{proof}
\begin{remark}
\label{RCF}
i) One observes that, for the generalized KdV
equation, we have
$a(r)=2r$ and for the nonlinear Schr\"odinger equation $a(r)=r$.\\
ii) In order to obtain the estimate \eqref{ap-4} for the Cauchy
problem \eqref{IVPgeral}, we multiply \eqref{IVPgeral} by
$|x|^{2r} \overline{u}$, take the real part and integrating we
have
\begin{align}\label{persxaverw}
0=\partial_t \int |x|^{2r}|u|^2 dx+ 2 \textrm{Re}\, \int (x
\cdot\overline{x})^r \overline{u} \,Lu dx + 2 \textrm{Re}\,
\int\left( |x|^{r} \overline{u}\right) |x|^{r} F(u,\nabla_x u) dx.
\end{align}
Then, by Theorem \ref{teorlinearpers} we only need to estimate the
third term in \eqref{persxaverw} (for the non-linear Schr\"odinger
this term is zero). Using Cauchy-Schwartz inequality
\begin{align}
2 \textrm{Re}\, \int\left( |x|^{r} \overline{u}\right) |x|^{r}
F(u,\nabla_x u) dx \le 2 \| |x|^{r}
\overline{u}\|_{L_x^2}\||x|^{r} F(u,\nabla_x u)\|_{L_x^2}.
\end{align}
Thus we need an estimate of the following form
\begin{align}
\||x|^{r} F(u,\nabla_x u)\|_{L_x^2}\le C  \| |x|^{r}
\overline{u}\|_{L_x^2} \; A(\|u\|_{H^{a(r)}}).
\end{align}
and it is possible if for example $F(x,y)=x \, G(x,y)$, where $G$
is a polynomial function and $a(r)>n/2+1$, in order to use
immersion of $u$ and $\nabla_x u$ in $L_x^{\infty}$ and therefore
$|G(u,\nabla_x u)| \le  A(\|u\|_{H^{a(r)}})$.
\end{remark}

\bigskip
\section{The Generalized Interpolation Lemma} \label{CL}

In this section we generalize the Abstract Interpolation Lemma
established by the authors in \cite{CN1}. In fact, we extend in
two directions: First, we generalize to multi-dimensional setting.
The second extension is concerned with the exponent $\theta$ of
the weight.

\medskip
Let $s> n/2$, $r \geq 1$ be fixed. For each $T>0$, we consider a
family $\clg{A}$ of functions $f$ from $[-T,T]$ in $H^s(\R^n)$,
satisfying the following conditions:

\bigskip
$(C1)$ The measure $\, \clg{L}^n\big( \{\xi \in \R^n; f(t,\xi)
\neq 0 \}\big)$ is positive, where $\mathcal{L}^n \big(E\big)$ is
the Lebesgue measure of a measurable set $E \subset \R^n$.

\medskip
 $(C2)$ There exist positive constants $C_0, \td{C_0}$ and a function $A_0 \geq 0$,
 which do not depend on $f$ and $t$, such that
\begin{eqnarray}
   \label{AIL0}
   \|f(t)\|^2_{L^2} \leq C_0 \; \|f(0)\|^2_{L^2},
%
\\[5pt]
  \label{AIL1}
   \|f(t)\|^2_{L^2(d\dot{\mu}_r)} \leq \td{C_0} \,
   \|f(0)\|^2_{L^2(d\dot{\mu}_r)} + A_0(\|f(0)\|_{{H}^{{a(r)}}}).
\end{eqnarray}

$(C3)$ For all $\theta \in [0,r]$, there exists $\Theta>0$, which
does not depend on $f$ and $t$, and $\gamma_1 \in (0,1/2)$, such
that
\begin{equation}
\label{AILC3}
    \int_{\{|f(t)|^2 < \Theta\}} |f(t)|^2 \; d\dot{\mu}_\theta
    \leq \gamma_1 \int_{\R^n} |f(t)|^2 \;
    d\dot{\mu}_\theta.
\end{equation}

\medskip
$(C4)$ There exist $R>0$ and $\gamma_2 \in (0,1)$ (both
independent of $f$), such that
\begin{equation}
\label{AILC4}
    \int_{\R^n \setminus B(0,R)} |f(0)|^2 \; d\dot{\mu}_r
    \leq \gamma_2 \int_{\R^n} |f(0)|^2 \;
    d\dot{\mu}_r.
\end{equation}

Clearly the set $\clg{A}$ depend on the constants $C_0$,
$\td{C_0}$, $R$, $\gamma_2$, and also the functions $A_0$,
$\Theta(\theta)$.
In the following, we present two different families, which satisfy
the conditions $(C1)$--$(C4)$. The former example, is a non
enumerable set of functions, which are not necessarily solutions
of a partial differential equation. On the other hand, the
elements of the family in the second example are solutions of the
dispersive equation \eqref{IVPgeral}.

\begin{example}
Let $R_0,T>0$, $r \ge 1$  be constants and $b>0$, such that, for
each $\theta \in [0,r]$,
\begin{equation}
\label{eq12} \int_{\{R_0 \le |\xi|\le R_0+b\}} |\xi|^{2
\theta}d\xi \le \dfrac{1}{3(T+1)^2}\int_{\{ |\xi|\le R_0\}}
|\xi|^{2 \theta}d\xi.
\end{equation}
Let $\clg{B}_0$ be the set of continuous functions in $\R^n$, such
that
$$
 g(\xi)=\begin{cases} 0, \quad \textrm{if}\,\, |\xi|> R_0+b, \\
 L, \quad \textrm{if}\,\, |\xi| \le R_0,
\end{cases}
$$
and $0\le g(\xi) \le L$, where $L$ is any positive real number,
fixed. Now, we set
 $$
    \clg{B}_1= \{f(t,\xi)= g(\xi)(1+|t|); t \in [-T,T], g \in \clg{B}_0 \}.
 $$
Then, the family $\clg{B}_1$ satisfies the $(C1)$--$(C4)$
conditions. Indeed, condition $(C1)$ is clearly satisfied. The
condition $(C2)$ is satisfied with $C_0=\tilde{C_0}=1+T$. The
condition $(C4)$ is satisfied with $R=R_0+b$ for all $\gamma_2 \in
(0,1)$, since the first integral in \eqref{AILC4} is null. And the
condition $(C3)$ is satisfied with $\Theta=L^2$ and
$\gamma_1=1/3$, since \eqref{eq12} implies
 \begin{align*}
 \int_{\{ |f(t)|^2< L^2\}} |\xi|^{2 \theta}|f(t,\xi)|^2 d\xi \le &
 \,(1+T)^2 L^2  \int_{\{ R_0 \le |\xi|\le R_0+b \}} |\xi|^{2 \theta}d\xi\\
 \le &\, \dfrac{(1+T)^2 L^2  }{3(1+T)^2}\int_{\{ |\xi|\le R_0 \}} |\xi|^{2 \theta} d\xi\\
 = & \, \dfrac13 \int_{\{ |\xi|\le R_0 \}} |\xi|^{2 \theta} |g(\xi)|^2d\xi\\
 \le & \, \dfrac13 \int_{\{ |\xi|\le R_0 \}} |\xi|^{2 \theta} |f(t,\xi)|^2d\xi\\
 \le & \, \dfrac13 \int_{ \R^n } |\xi|^{2 \theta}
 |f(t,\xi)|^2d\xi.
 \end{align*}
\end{example}

We remark that, the following example will be used in the proofs
of Theorems \ref{teomain} and \ref{teomain1}.

\begin{example}
\label{EXAMPLETWO} We consider the evolution equation
\eqref{IVPgeral} under the conditions on Remark \ref{RCF}.
We assume that,
\begin{equation}
\label{u0} u_0(x) \in \mathcal{X}^{s,r}, \quad u_0 \neq 0.
\end{equation}
Now, let $(u_0^{k})$ be a sequence of regular functions $($in
$S(\R^n)$ or in some $\mathcal{X}^{s_0,\theta}$, with $s\ll s_0)$,
such that
\begin{equation}
\label{converg1p}
 u_0^{k}
\rightarrow u_0 \quad \textrm{in}\,\, \mathcal{X}^{s,\theta}, \quad \textrm{when}\,\, k \to \infty.
\end{equation}
If the IVP \eqref{IVPgeral} satisfies the conditions \eqref{ap-1}--\eqref{ap-4},
and it is well-posed in 
$C([-T,T];H^s) $, then the set of solutions
\begin{equation}
  \mathcal{C}= (u^{k}(t)), \quad \big(k >N_0, \; \textrm{for some}
\,\,N_0>0 \big),
\end{equation}
of the IVP \eqref{IVPgeral} with initial data $u_0^{k}$, satisfies
the conditions $(C1)$--$(C4)$. Indeed, we have the following:

\medskip
{\underline{ Condition $(C1)$}}

We prove this by contradiction. First, we suppose that
$$
  \forall N_0, \,\,\exists k \ge N_0,\,\, \exists t \in [-T,T]; \,\,\, \clg{L}^n\big( \{x \in \R^n; u^{k}(t,x) \neq
0 \}\big)=0.
$$
Then, there exist $k_m \ge m$, $m=1,2\cdots ,$ and $t_m\in
[-T,T]$, such that
\begin{equation}
\label{zero}
   u^{k_m}(t_m,x)=0, \,\, x - a.e.
\end{equation}
By \eqref{converg1p}, \eqref{zero} and the continuous dependence
of the initial data, the sequence of solutions $u^{k}(t)$,
associated to IVP \eqref{IVPgeral} and initial data $u_0^{k}$,
satisfy
\be\label{converg11pers} \|
 u(t_m)\|_{H^s} = \|u^{k_m}(t_m)
 -u(t_m)\|_{H^s} \le \sup_{t \in [-T,T]} \|u^{k_m}(t)
 -u(t)\|_{H^s} \stackrel{m \to \infty}{\rightarrow} 0.
 \ee

As $u \in C([-T,T];H^s)$, $t_m \in [-T,T]$, by compactness we can
assume that, $t_m \to t_0 \in [-T,T]$. Thus $\|
 u(t_m)\|_{H^s} \to \|
 u(t_0)\|_{H^s} $  and by \eqref{converg11pers}, it follows that $ \|
 u(t_0)\|_{H^s} =0$, which implies
$$u(t_0, x)=0, \quad x- a.e.$$
By uniqueness of solutions, we have for any $t \in [-T,T]$,
$u(t,x)=0$ almost everywhere. In particular, $u(0,x)= u_0=0$,
which is to say a contradiction with \eqref{u0}.

\medskip
{\underline{ Condition $(C2)$}}

It is a direct consequence given from the fact that, the solution
$u$ of the IVP \eqref{IVPgeral} satisfies the conditions
\eqref{ap-1} and \eqref{ap-4}.

\medskip
{\underline{ Condition $(C3)$}}

We must be prove that
$$
  \forall \theta \in [0,r],\,\, \exists \Theta>0, \,\,
  \textrm{s.t.}\,\, \forall k > N_0,\,\,
  \forall t \in [-T,T] \,\, \textrm{and for some} \,\, \gamma_1 \in
  (0,1/2),
$$
we have
\begin{equation}
\label{AILC3E}
    \int_{\{|u^{k}(t)|^2 < \Theta\}} |u^{k}(t)|^2 \; d\dot{\mu}_\theta
    \leq \gamma_1 \int_{\R^n} |u^{k}(t)|^2 \;
    d\dot{\mu}_\theta.
\end{equation}
Again, we prove this condition by contradiction. We suppose that
$$
\exists \theta \in [0,r],\,\, \forall \Theta>0, \,\, \exists k > N_0,\,\, \exists t \in [-T,T] \,\, \textrm{and} \,\, \forall \gamma_1 \in (0,1/2)
$$
and we have
\begin{equation*}
    \int_{\{|u^{k}(t)|^2 < \Theta\}} |u^{k}(t)|^2 \; d\dot{\mu}_\theta
    > \gamma_1 \int_{\R^n} |u^{k}(t)|^2 \;
    d\dot{\mu}_\theta.
\end{equation*}
Then, there exist $k_j >N_0$, $t_j \in [-T,T]$, $j \in Z^+$ and
$\gamma_0 \in (0,1/2)$, such that
\begin{equation}
\label{AILC3pers}
    \int_{\{|u^{k_j}(t_j)|^2 < 1/j\}} |u^{k_j}(t_j)|^2 \; d\dot{\mu}_\theta
    > \gamma_0 \int_{\R^n} |u^{k_j}(t_j)|^2 \;
    d\dot{\mu}_\theta.
\end{equation}
Now, without lost of generality, we can suppose that
\be \label{convert0} t_j \to t_0 \in [-T,T], \quad
\textrm{when}\,\, j \to \infty.
\ee
Further, we consider the map
\be\label{aplic}\Phi :\Z^+ \to \clg{V}=\{k_j;\, j \in \Z^+\},
\quad\Phi(j)=k_j, \ee
 and the following cases:
\\
{\underline{ Case I, the set $\clg{V}$ is not finite}}: In this
case we can suppose that $k_j \to\infty$, when $j \to \infty$. By
the immersion ($s>n/2$), \eqref{converg1p} and continuous
dependence of the initial data, the subsequence of solutions
$u^{k_j}(t_j)$, associated to IVP \eqref{IVPgeral} and initial
data $u_0^{k_j}$, satisfy
\begin{align}\label{converg1pers}
|u^{k_j}&(t_j,x)- u(t_0,x)| \le C \|u^{k_j}(t_j)
 -u(t_j)\|_{H^s} +C \|u(t_j)
 -u(t_0)\|_{H^s}\nonumber\\
&\le \sup_{t \in [-T,T]} \|u^{k_j}(t)
 -u(t)\|_{H^s}+C \|u(t_j)
 -u(t_0)\|_{H^s} \stackrel{j \to \infty}{\longrightarrow} 0,
 \end{align}
 where we have used that $u \in C([-T,T];H^s)$.

Then, using \eqref{converg1pers}, the Dominated Convergence
Theorem and \eqref{AILC3pers}, we obtain a contradiction.
\\
{\underline{ Case II, the set $\clg{V}$ is finite}}: In this case,
concerning the application $\Phi$, there exists $k_q \in \clg{V}$,
such that, $\clg{V}_0:= \Phi^{-1}\{k_q\}=\{ q_1,q_2, \cdots \}
\subseteq \Z^+$ must be not finite, with $q_{j}<q_{j+1} $, for
each $j \in \Z^+$. Therefore, by \eqref{AILC3pers} we get
\begin{equation}
\label{AILC3pers111}
    \int_{\{|u^{k_q}(t_{q_j})|^2 < 1/q_j\}} |u^{k_q}(t_{q_j})|^2 \; d\dot{\mu}_\theta
    > \gamma_0 \int_{\R^n} |u^{k_q}(t_{q_j})|^2 \;
    d\dot{\mu}_\theta.
\end{equation}
If $j \to \infty$, then $q_j \to \infty$ and by \eqref{convert0}
$t_{q_j} \to t_0$. As $u^{k_q} \in C([-T,T];H^s)$, by the
immersion, we have for any $x\in \R^n$
\be
 |u^{k_q}(t_{q_j},x)
 -u^{k_q}(t_0,x)| \le C \|u^{k_q}(t_{q_j})
 -u^{k_q}(t_0)\|_{H^s} \stackrel{j \to \infty}{\longrightarrow} 0.
\ee
Therefore, arguing as previous in Case I and taking the limit in
\eqref{AILC3pers111}, we obtain a contradiction.

\medskip
{\underline{ Condition $(C4)$}}

We prove that: There exist $R>0$ and $\gamma_2 \in (0,1)$, such
that for any $k>N_0$
\begin{equation*}
    \int_{\R^n \setminus B(0,R)} |u^{k}(0,x)|^2 \;|x|^{2r} dx
    \leq \gamma_2 \int_{\R^n} |u^{k}(0,x)|^2 \;
    \;|x|^{2r} dx.
\end{equation*}
Again by contradiction, we suppose that
$$
\forall R>0, \,\, \forall \gamma_2 \in (0,1),\,\, \exists
k>N_0,\,\,\textrm{such that}
$$
\begin{equation}
\label{AILC4pers}
    \int_{\{|x|>R\}} |u^{k}(0,x)|^2 \;|x|^{2r} dx
    >\gamma_2 \int_{\R^n} |u^{k}(0,x)|^2 \;
    \;|x|^{2r} dx.
\end{equation}
In particular, this proposition implies for any $m \in \Z^+$,
there exists $k_m>N_0$, such that
\begin{equation}
\label{AILC4pers1}
    \int_{\{|x|>m\}} |u_0^{k_m}(x)|^2 \;|x|^{2r} dx
    >\left(1-\frac1m\right) \int_{\R^n} |u_0^{k_m}(x)|^2 \;
    \;|x|^{2r} dx.
\end{equation}
Let us consider the map \be\label{aplic}\Gamma :\Z^+ \to
\clg{W}=\{k_m;\, m \in \Z^+\}, \quad\Gamma(m)=k_m, \ee
 and the following cases:
\\
{\underline{ Case I, the set $\clg{W}$ is not finite}}: In this
case we can suppose that $k_m \to\infty$, when $m \to \infty$ and
thus by \eqref{converg1p}, we obtain
\be\label{persx0}
\int_{\{|x|>m\}} |u_0^{k_m}(x) - u_0(x)|^2 \;
    \;|x|^{2r} dx \le \int_{\R^n} |u_0^{k_m}(x) - u_0(x)|^2 \;
    \;|x|^{2r} dx \stackrel{m \to \infty}{\rightarrow} 0
\ee
and
\be\label{persx1}
\int_{\R^n} |u_0^{k_m}(x)|^2 \;
    \;|x|^{2r} dx \stackrel{m \to \infty}{\rightarrow} \int_{\R^n} |u_0(x)|^2 \;
    \;|x|^{2r} dx.
\ee Moreover, as
\begin{align}\label{persx2}
 \int_{\{|x|>m\}} |u_0^{k_m}(x)|^2 \;|x|^{2r} dx  &\le \int_{\{|x|>m\}} |u_0^{k_m}(x) - u_0(x)|^2 \;
    \;|x|^{2r} dx \nonumber\\
& +\int_{\{|x|>m\}} |u_0(x)|^2 \;
    \;|x|^{2r} dx \stackrel{m \to \infty}{\longrightarrow} 0,
\end{align}
from \eqref{AILC4pers1}--\eqref{persx2}, we arrive to a
contradiction.
\\
{\underline{ Case II, the set $\clg{W}$ is finite}}: In this case,
again concerning the application $\Gamma$, there exists $k_p \in
\clg{W}$, such that $\clg{W}_0:= \Gamma^{-1}\{k_p\}=\{ p_1,p_2,
\cdots \} \subseteq \Z^+$ is not finite, with $p_{i}<p_{i+1} $, $i
\in \Z^+$. Therefore, by \eqref{AILC4pers1} we get
\begin{equation}\label{persisxav3}
\int_{\{|x|>p_m\}} |u_0^{k_p}(x)|^2 \;|x|^{2r} dx
    >\left(1-\frac{1}{p_m}\right) \int_{\R^n} |u_0^{k_p}(x)|^2 \;
    \;|x|^{2r} dx.
\end{equation}
Similarly to Case I before, taking the limit in \eqref{persisxav3}
when $m \to \infty$  ($p_m \to \infty$), we obtain a
contradiction.
\end{example}

Now we pass to the Generalized Abstract Interpolation Lemma.
\begin{lemma}\label{G1CL}
Let $r \geq 1$ be a real number, and $\clg{A}$ a family satisfying
the conditions $(C1)$--$(C4)$. Then, for each $\theta \in (0,r)$,
there exists a positive constant $\rho(\theta,r)$, such that, for
each $t \in [-T,T]$,
\be \label{AIL}
    \|f(t)\|^2_{L^2(d\dot{\mu}_\theta)} \leq \|f(t)\|^{2 \rho}_{H^s} \;
    \Big( K_0 \, \|f(0)\|^2_{L^2}
    + K_1 \, \|f(0)\|^2_{L^2(d\dot{\mu}_\theta)}
    + K_2 \Big)
\ee
for all $f \in \clg{A}$, where
$$
  K_0= C_0 \,  R^{2 \theta} \, \left(\frac{4}{\Theta}\right)^{\rho +
  1},\quad
  K_1= \frac{\tilde{C}_0}{\rho (1-\gamma_2)} \,
  \left(\frac{4}{\Theta}\right)^{\rho}, \quad
  K_2= \frac{A_0(\|f(0)\|_{{H}^{{a(r)}}})}{\rho R^{2 \theta \rho}}.
$$
\end{lemma}

\begin{proof}
For simplicity, we write $f(t,\xi) \equiv f(\xi)$ and $f(0,\xi)
\equiv f_0(\xi)$. Let $\kappa_j>0$, $(j=0,1)$, be constants
independents of $t$, and for $\theta \in [0,r]$, we set
$$
\begin{aligned}
I_1^{\kappa_1}:=& \int_{\R^n} \|\xi\|^{2 \theta} \, |f(\xi)|^2 \,
\chi_{\{|f(\xi)|^2>\kappa_1\}} \; d\xi,
\\
I_2^{\kappa_1}:=& \, \kappa_1 \int_{\R^n} \|\xi\|^{2 \theta} \,
\chi_{\{|f(\xi)|^2> \kappa_1\}} \; d\xi,
\\
I_3^{\kappa_1}:=& \int_{\R^n} \|\xi\|^{2 \theta} \, |f(\xi)|^2 \,
\chi_{\{|f(\xi)|^2 \leq \kappa_1\}} \; d\xi,
\end{aligned}
$$
where $\chi_E$ is the characteristic function of the set $E$.
Then, we have
$$
  I:= \int_{\RR^n} \|\xi\|^{2 \theta} \,
  |f(\xi)|^2  \; d\xi = I_1^{\kappa_1} + I_3^{\kappa_1} =
  I_1^{\kappa_1} - \kappa_0 I_2^{\kappa_1} + I_3^{\kappa_1} +
  \kappa_0 I_2^{\kappa_1}.
$$
It is not difficult to show that $I_2^{\kappa_1} <
I_1^{\kappa_1}$,
hence $\kappa_0 I_2^{\kappa_1} <
\kappa_0 \, (I_1^{\kappa_1}+I_3^{\kappa_1}) = \kappa_0 \, I$.
Consequently, we have
\begin{equation}
\label{I11}
      (1 - \kappa_0) \, I < I - \kappa_0 I_2^{\kappa_1} =
      I_1^{\kappa_1} - \kappa_0 I_2^{\kappa_1}+ I_3^{\kappa_1}.
\end{equation}

Now, we show that, there exist $\theta_1>0$ independent of $f$, $t
\in [-T,T]$, and a positive constant $\beta <1$, such that
$I_3^{\kappa_1} < \beta I_1^{\kappa_1}$. Indeed, we have
$$
  \begin{aligned}
  \int_{\R^n} \|\xi\|^{2 \theta} \, |f(\xi)|^2 \,
  \chi_{\{|f|^2 \leq \kappa_1\}} \; d\xi
  &\leq \beta \,
  \int_{\R^n} \|\xi\|^{2 \theta} \,
  |f(\xi)|^2 \, \chi_{\{|f|^2 > \kappa_1\}} \; d\xi
\\
  &= \beta \,
  \int_{\R^n} \|\xi\|^{2 \theta} \,
|f(\xi)|^2 \; d\xi
\\
  &- \beta \,
  \int_{\R^n} \|\xi\|^{2 \theta} \,
|f(t,\xi)|^2 \, \chi_{\{|f|^2 \leq \kappa_1\}} \; d\xi,
\end{aligned}
$$
and hence, we must have
$$
  \begin{aligned}
    \int_{\R^n} \|\xi\|^{2 \theta} \, |f(\xi)|^2 \,
  \chi_{\{|f|^2 \leq \kappa_1\}} \; d\xi
  \leq \frac{\beta}{1 + \beta} \; \int_{\R^n} \|\xi\|^{2 \theta} \,
|f(\xi)|^2 \; d\xi,
\end{aligned}
$$
which is satisfied since $f \in \clg{A}$. Consequently, we take
$\kappa_1= \Theta$ in inequality \eqref{AILC3}. One observes that,
since $\beta <1$, it follows that $\beta/(1+\beta) < 1/2$. It
follows that, there exists a positive constant $\alpha < 1/2$,
such that
\begin{equation} \label{I12}
    I_3^{\kappa_1} < \alpha (I_1^{\kappa_1} + I_3^{\kappa_1}) = \alpha I.
\end{equation}
Hence we fix $\kappa_0 = (3/4 - \alpha)>1/4$ and, from
\eqref{I11}, \eqref{I12}, we obtain
\begin{equation} \label{I13} I < \frac{I_1^{\kappa_1} - \kappa_0
I_2^{\kappa_1}}{1 - (\kappa_0 + \alpha)}=4\left(I_1^{\kappa_1} -
\kappa_0 I_2^{\kappa_1}\right).
\end{equation}

At this point, we claim that, there exist $N_1 \in \NN$ and a
constant $C_1>0$ both independent of $f$ and $t$, such that, for
all $\eta \geq N_1$
$$
  \int_{\{ \|\xi\| < \eta\}} |f(\xi)|^2 \|\xi\|^{2r} \; d\xi \leq C_1
  \int_{\{ \|\xi\| < \eta\}} |f_0(\xi)|^2 \|\xi\|^{2r} \; d\xi+\td{C_1}.
$$
In order to prove the claim, we show that
$$
  \begin{aligned}
  \int_{\R^n} |f(\xi)|^2 \|\xi\|^{2r} \; d\xi &- \int_{\{ \|\xi\| \geq \eta\}} |f(\xi)|^2 \|\xi\|^{2r} \; d\xi
  \\
  &\leq C_1 \int_{\R^n} |f_0(\xi)|^2 \|\xi\|^{2r} \; d\xi -
  C_1 \int_{\{\|\xi\| \geq \eta\}} |f_0(\xi)|^2 \|\xi\|^{2r} \;
  d\xi+\td{C_1},
  \end{aligned}
$$
for each $\eta \geq N_1$. Therefore, from \eqref{AIL1} and
supposing $C_1> \td{C_0}$, it is enough to show that
$$
  \begin{aligned}
  \td{C_1} &+ \td{C_0} \int_\R |f_0(\xi)|^2 |\xi|^{2r} \; d\xi -
  \int_{\{|\xi| \geq \eta\}} |f(\xi)|^2 |\xi|^{2r} \; d\xi
  \\
  &\leq C_1 \int_\R |f_0(\xi)|^2 |\xi|^{2r} \; d\xi -
  C_1 \int_{\{|\xi| \geq \eta\}} |f_0(\xi)|^2 |\xi|^{2r} \;
  d\xi + \td{C_1}.
  \end{aligned}
$$
By a simple algebraic manipulation, it is sufficient to show that
$$
  \int_{\{\|\xi\| \geq \eta\}} |f_0(\xi)|^2 \|\xi\|^{2r} \;
  d\xi \leq \frac{C_1-\td{C_0}}{C_1} \int_{\R^n} |f_0(\xi)|^2 \|\xi\|^{2r} \;
  d\xi,
$$
which is true for each $f \in \clg{A}$, and we take $N_1= R$ of
inequality \eqref{AILC4}.

\bigskip
Now, we proceed to estimate $I_1^{\kappa_1} - \kappa_0
I_2^{\kappa_1}$. If $\theta \in \{0,r\}$, then by \eqref{AIL0} and
\eqref{AIL1}, it is obvious that
$$
I_1^{\kappa_1} - \kappa_0 I_2^{\kappa_1} \le C_0 \int_{\R^n}
\|\xi\|^{2 \theta} \, |f_0(\xi)|^{2} \; d\xi.
$$
Then, we consider in the following $\theta \in (0,r)$. Denoting
$\kappa= (\kappa_0 \kappa_1)^{1/2 \theta}$, it follows that
$$
\begin{aligned}
      I_1^{\kappa_1} - \kappa_0 I_2^{\kappa_1}& = \int_{\R^n} \big(\|\xi\|^{2 \theta} \,
      |f(\xi)|^2 - \kappa^{2 \theta} \, \|\xi\|^{2 \theta} \big) \,
\chi_{\{|f(\xi)|^2> \kappa_1\}} \; d\xi
\\
      &= \int_{\R^n} \big(\,(\, \|\xi\| \,
      |f(\xi)|^{1/\theta} \,)^{2 \theta} - \kappa \, \|\xi\| \,)^{2 \theta} \big) \,
\chi_{\{|f(\xi)|^2> \kappa_1\}} \; d\xi \\
      & = \int_{\R^n} \int_{\kappa \,
      \|\xi\|}^{\|\xi\| \, |f(\xi)|^{1/\theta}}
      \varphi'(\eta) \,  d\eta \; d\xi
\\
      &= 2 \theta \int_0^\infty \eta^{2 \theta-1} \; \mathcal{L}^n\big(E(\eta)\big) \,
      d\eta,
\end{aligned}
$$
where for each $\eta>0$, $\varphi(\eta) = \eta^{2 \theta}$ and
$$
  E(\eta):= \big\{ \xi \in \RR^n \,/\,
  |f(\xi)|^{1/\theta} \|\xi\| > \eta \, \big\} \bigcap
  \big\{ \xi \in \RR^n \,/\, \, \kappa \, \|\xi\| < \eta  \big\}.
$$
One observes that, for each $\eta>0$, $E(\eta) \not= \emptyset$
(in the geometric measure sense). Indeed, assume for contrary that
$E(\eta)= \emptyset$, then $\clg{L}^n\big(E(\eta)\big)= 0$ and
thus $I<0$ from \eqref{I13}, which is a contradiction by condition
$(C1)$, and the definition of $I$. Moreover, we observe that,
since
$$
  1 < \frac{|f(\xi)|^{2r/\theta} \, \|\xi\|^{2r}}{\eta^{2r}},
$$
we could write
$$
\begin{aligned}
  \mathcal{L}^n\big( E(\eta) \big) &\leq \int_{\{\|\xi\| < \eta/\kappa\}}
  \frac{|f(\xi)|^{2r/\theta} \|\xi\|^{2r}}{\eta^{2r}}
  \; d\xi.
\end{aligned}
$$
%
%
%
Therefore, we have
$$
\begin{aligned}
      I_1^{\kappa_1} - \kappa_0 I_2^{\kappa_1}& \leq  2 \theta
      \int_0^\infty \eta^{2 \theta-1}
      \int_{\{\, \|\xi\| < \eta / \kappa \}} \frac{|f(\xi)|^{2r/\theta} \, \|\xi\|^{2r}}{\eta^{2r}} \; d\xi \, d\eta
\\[7pt]
      &\leq 2 \theta \; \|f(t)\|_{H^s}^{(2r/\theta) - 2} \int_0^\infty \eta^{2 \theta - 2r - 1}
      \int_{\{\, \|\xi\| < \eta / \kappa \}} |f(\xi)|^{2} \, \|\xi\|^{2r}  \; d\xi \,
      d\eta,
\end{aligned}
$$
where we have used the Sobolev Embedding Theorem. Hence applying
\eqref{AIL0}, we obtain
$$
\begin{aligned}
      I_1^{\kappa_1} &- \kappa_0 I_2^{\kappa_1} \le 2 \theta \; \|f(t)\|_{H^s}^{(2r/\theta) - 2}
      \int_0^{N_1} \eta^{2 \theta-2r-1}
      \int_{\{\, \|\xi\| < \eta/\kappa \}} |f(\xi)|^{2} \, \frac{\eta^{2r}}{\kappa^{2r}} \;  \; d\xi \,
      d\eta
\\[5pt]
     &+ 2 \theta \; \|f(t)\|_{H^s}^{(2r/\theta) - 2}
      \int_{N_1}^\infty \eta^{2 \theta - 2r - 1}
      \int_{\{\, \|\xi\| < \eta/\kappa\}} |f(\xi)|^{2} \, \|\xi\|^{2r} \;  \; d\xi \,
      d\eta
\\[5pt]
      &\leq \frac{C_0 \;N_1^{2 \theta}}{\kappa^{2r}} \; \|f(t)\|_{H^s}^{(2r/\theta) - 2} \,
      \int_{\R^n} |f_0(\xi)|^{2} \; d\xi
\\[5pt]
      &+ 2 \theta \; C_1 \,\|f(t)\|_{H^s}^{(2r/\theta) - 2}  \,
      \int_{\R^n} |f_0(\xi)|^{2} \, \|\xi\|^{2r}
      \int_{\{ \eta > \kappa \, \|\xi\| \}} \eta^{2 \theta - 2r -1}  \;  \; d\eta \,
      d\xi + \Xi
\\[5pt]
      &\leq C_0 \left(\frac{\;4 }{\kappa_1}\right)^{r/\theta} \; N_1^{2\theta} \;
      \|f(t)\|_{H^s}^{(2r/\theta) - 2} \,
      \int_{\R^n} |f_0(\xi)|^{2} \; d\xi+
\\[5pt]
      &+ C_1 \left(\frac{4}{\kappa_1}\right)^{(r-\theta)/\theta}\, \frac{\theta}{r - \theta}
      \, \|f(t)\|_{H^s}^{(2r/\theta) - 2} \,
      \int_{\R^n} |f_0(\xi)|^{2}
      \; \|\xi\|^{2\theta} \,
      d\xi + \Xi,
\end{aligned}
$$
where
$$
  \Xi= \td{C_1} \frac{\theta}{r-\theta} \frac{\|f(t)\|_{H^s}^{(2r/\theta) - 2}}{N_1^{2(r-\theta)}}.
$$
\end{proof}
%
%

\section{ Statement of the well-posedness result}

This is the section where we prove the well-posedness of the
Cauchy problem \eqref{IVP} in weighted Sobolev space
$\clg{X}^{s,\theta}$, for $s \geq 2r$ and $\theta \in [0,r]$.

\bigskip
\subsection{Proof of Theorems \ref{teomain} and \ref{teomain1}}

\begin{proof}[{\bf Proof of Theorem \ref{teomain}}]

Consider $r \geq 1$, $u_0 \in \mathcal{X}^{s,\theta}$, $s \ge 2
r$, $\theta \in [0,r]$, with $u_0 \neq 0$. We know that there
exists an function $u \in C([-T,T], H^{s})$, such that the IVP
\eqref{IVP} is global well-posed in $H^{s}$. It is well known that
$\mathbf{S}(\R)$ is dense in $\mathcal{X}^{s,\theta}$. Therefore,
for $u_0 \in \mathcal{X}^{s,\theta}$ there exist a sequence
$(u_0^{\lambda})$ in $\mathbf{S}(\R)$ such that \be\label{converg}
u_0^{\lambda} \to u_0 \quad \textrm{in}\,\,
\mathcal{X}^{s,\theta}. \ee
By continuous dependence, the sequence of solutions
$u^{\lambda}(t)$ associated to IVP \eqref{IVP} and with initial
data $u_0^{\lambda}$ satisfy
\be\label{converg1}
 \sup_{t \in [-T,T]} \|u^{\lambda}(t)
 -u(t)\|_{H^s} \stackrel{\lambda \to \infty}{\rightarrow} 0.
 \ee
Now, assuming conditions \eqref{AC}, suppose temporarily that the
solutions $u^{\lambda}$ of the IVP
\begin{equation}\label{IVPmnx}
\begin{cases}
\p_t u^\lambda + a(u^\lambda) \partial_x u^\lambda + \partial_{x}^3 u^\lambda= 0, \quad (t,x) \in \R^2,\\
u^\lambda(x,0)  =  u_0^\lambda(x),
\end{cases}
\end{equation}
satisfy the conditions $(C1)$--$(C4)$ of Section \ref{CL}.
Therefore Lemma \ref{G1CL} gives
\begin{equation*}
\int_\RR |\xi|^{2 \theta} |u^{\lambda}(t,\xi)|^2 \, d\xi \leq C \;
\big(\int_\RR |u^{\lambda}(0,\xi)|^2 \,
  d\xi + \int_\RR |\xi|^{2 \theta} |u^{\lambda}(0,\xi)|^2 \,  d\xi + 1 \big),
\end{equation*}
where $C=  C
(\theta,\|u^{\lambda}(t)\|_{H^s},\|u^{\lambda}(0)\|_{L^2},\|u^{\lambda}_x(0)\|_{L^2},
\|u^{\lambda}_{xx}(0)\|_{L^2},T)$, 
taking the
limit when $\lambda \to \infty$, \eqref{converg1} implies
\begin{equation*}
\int_\RR |\xi|^{2 \theta} |u(t,\xi)|^2 \, d\xi \leq C \;
\big(\int_\RR |u(0,\xi)|^2 \,
  d\xi + \int_\RR |\xi|^{2 \theta} |u(0,\xi)|^2 \,  d\xi + 1 \big),
\end{equation*}
where $C=  C
(\theta,\|u(t)\|_{H^s},\|u(0)\|_{L^2},\|u_x(0)\|_{L^2},\|u_{xx}(0)\|_{L^2},T)$.
Thus $u(t) \in \mathcal{X}^{s,\theta}$, $t \in [-T,T]$, which
proves the persistence. The local well-posedness theory in $H^{s}$
implies the uniqueness, thanks for that, we obtain uniqueness in
$\mathcal{X}^{s,\theta}$.

Finally, following the proof in Example \ref{EXAMPLETWO}, we prove
that the sequence of solutions $(u^{\lambda_n}(t))$ satisfy the
conditions (C1)--(C4).
\end{proof}

\begin{proof}[{\bf Proof of Theorem \ref{teomain1}}]

By Theorem \ref{teomain} is sufficient prove continuous dependence
in the norm $\|\cdot\|_{L^2(d\dot{\mu}_{\theta})}$. Let $u(t)$ and
$v(t)$ be two solutions in $\mathcal{X}^{s,\theta}$, of the IVP
\eqref{IVP-1} with initial dates $u_0$ and $v_0$ respectively, let
$u^{\lambda}(t)$, $v^{\lambda}(t)$ be the solutions of the IVP
\eqref{IVP-1} with initial dates $u_0^{\lambda}$ and
$v_0^{\lambda}$ respectively such that $u_{0}^{\lambda},
v_{0}^{\lambda} \in \mathbf{S}(\R)$, $u_{0}^{\lambda} \to u_0$,
$v_{0}^{\lambda} \to v_0$ in $\mathcal{X}^{s,\theta}$ and with
$\lambda >> 1$, we have
\begin{align*}
\|u(t)-v(t)\|_{L^2(d\dot{\mu}_{\theta})}\le & \|u(t)-u^{\lambda}(t)\|_{L^2(d\dot{\mu}_{\theta})}+\|u^{\lambda}(t)-v^{\lambda}(t)\|_{L^2(d\dot{\mu}_{\theta})}\\
&+\|v^{\lambda}(t)-v(t)\|_{L^2(d\dot{\mu}_{\theta})}.
\end{align*}
 Convergence in \eqref{converg1}
 implies for $\lambda >>1$ that
 $$|u(x,t)-u^{\lambda}(x,t)|\le 2 |u(x,t)| \quad \textrm{and} \quad \quad |v(x,t)-v^{\lambda}(x,t)|\le 2 |v(x,t)|,$$
and the Dominated Convergence Lebesgue's Theorem gives
 \begin{align*}\|u(t)-u^{\lambda}(t)\|_{L^2(d\dot{\mu}_{\theta})}\to 0 \quad \textrm{and} \quad\|v^{\lambda}(t)-v(t)\|_{L^2(d\dot{\mu}_{\theta})}\to 0.
\end{align*}
Let $w^\lambda:=u^{\lambda}-v^{\lambda}$, then $w^\lambda$
satisfies the equation
\begin{align*}
w^\lambda_t &+
w^\lambda_{xxx}+(u^{\lambda})^{k}w^\lambda_{x}+v^{\lambda}_{x}A(u^{\lambda}, u^{\lambda}){w}^\lambda_{x}=0,
\end{align*}
where $A(x,y)=x^{k-1}+x^{k-2}y+\cdots +x y^{k-1}+y^{k-1}$.

Then, we multiply the above equation by $\bar{w}^\lambda$,
integrate on $\R$ and take two times the real part, to obtain
$$
  \partial_t \int_\R |w^\lambda(t,x)|^2 \, dx \leq
  h(\|u_0\|_{H^2},\|v_0\|_{H^2}) \,
  \int_\R |w^\lambda(t,x)|^2 \, dx,
$$
where $h$ is a polynomial function with $h(0,0)= 0$ and we have used \eqref{ap-1}-\eqref{ap-4} and convergence \eqref{converg}. Therefore, by
Gronwall's Lema, we have
$$
  \|w^\lambda(t)\|_{L^2} \leq \exp\big(T \,
  h(\|u_0\|_{H^2},\|v_0\|_{H^2})\big) \, \|w^\lambda_{0}\|_{L^2},
$$
which gives the continuous dependence in case $\theta= 0$.
Moreover, when $\theta=r$ with an analogously argument as used in
the proof of Theorem \ref{kato} in Section \ref{IN}.

$$
  \begin{aligned}
  \|w^\lambda(t)\|_{L^{2}(d\dot{\mu}_r)} \leq \exp& \big(T \,
  h_1(\|u_0\|_{H^{2r}},\|v_0\|_{H^{2r}})\big)
  \\
  &\times \,
  \Big(\|w^\lambda_{0}\|_{L^{2}(d\dot{\mu}_r)} + h_1(\|u_0\|_{H^{2r}},\|v_0\|_{H^{2r}}) \Big),
  \end{aligned}
$$
where $h_1$ is a continuous function with $h_1(0,0)=0$.

Consequently, applying the Abstract Interpolation Lemma, we obtain
the continuous dependence for $\theta \in (0,r)$, where we have
assumed that the family $(w^\lambda)$ satisfies the hypothesis of
the Abstract Interpolation Lemma. Indeed, these properties for the
family $(w^\lambda)$ are demonstrated in a similar way done in the
proof of Theorem \ref{teomain}.
\end{proof}

\section*{Acknowledgements}

The authors were partially supported by FAPERJ through the grant
E-26/ 111.564/2008 entitled {\sl ``Analysis, Geometry and
Applications''}, and Pronex-FAPERJ through the grant E-26/ 110.560/2010 entitled \textsl{%
"Nonlinear Partial Differential Equations"}. The first author is
also partially supported by the National Counsel of Technological
and Scientific Development (CNPq)by the grant 303849/2008-8.

\newcommand{\auth}{\textsc}



\begin{thebibliography}{99}

\addcontentsline{toc}{chapter}{Bibliografia}
\bibitem{B-L} J. Berg, J. Lofstrom, \emph{Interpolation Spaces}, Springer, Berlin, 1976.

\bibitem{CN1} X. Carvajal and W. Neves, {\em Persistence of solutions to higher order
nonlinear Schr\"odinger equation}, J. Diff. Equations, v. 249,
(2010), 2214-2236.



\bibitem{HNT} N. Hayashi, K. Nakamitsu, M. Tsutsumi, {\em Nonlinear Schr\"odinger equations in weighted Sobolev spaces}, Funkcialaj Ekvacioj, {\bf 31} (1988) 363--381.

\bibitem{KT1}  T. Kato, \emph{On the Cauchy Problem for the
(Generalized) Korteweg-de-Vries Equation}, Studies in Applied
Mathematics, Advances in Math. Supplementary Studies, {\bf 08}
(1983), 93--127.

\bibitem{K-M}  T. Kato and K. Masuda, \emph{Nonlinear Evolution Equations and Analycity. I}, {\bf 03} (1986), 455--467.

\bibitem{N10} J. Nahas, \emph{A decay property of solutions to the
$k$-generalized KdV equation}, arXiv:1010.5001v2.

\bibitem{N8} J. Nahas, \emph{ A decay property of solutions to
the mKdV equation}, PhD. Thesis University of California-Santa
Barbara, June 2010.

\bibitem{NP9} J. Nahas, G. Ponce, \emph{On the persistent
properties of solutions of nonlinear dispervise equations in
weighted Sobolev spaces}. Harmonic analysis and nonlinear partial
differential equations, 23--36, RIMS Kôkyûroku Bessatsu, B26, Res.
Inst. Math. Sci. (RIMS), Kyoto, 2011.

\bibitem{NP} J. Nahas, G. Ponce, \emph{On the persistent
properties of solutions to semi-linear Schr\"odinger equation},
Comm. Partial Diff. Eqs., {\bf 34} (2009), no. 10--12,1208--1227.

\bibitem{[O-T]} T. Osawa, Y. Tsutsumi, \emph{Space-time estimates for null gauge
forms and nonlinear Schr\"odinger equations}, Differential
Integral Equations, {\bf 11} (1998), 201-222.

\bibitem{[P-S-S-M]} K. Porsezian, P. Shanmugha, K. Sundaram and A. Mahalingam, Phys. Rev. 50E,1543 (1994).
\bibitem{[SS]}C. Sulem, P. L. Sulem \emph{The nonlinear Schr\"odinger equation: sel-focusing and wave collapse},
Applied Mathematical Scienses, Springer Verlag {\bf 139} (1999),
350 pages.
\bibitem{T1} H. Takaoka, \emph{Well-posedness for the one dimensional
Schr\"odinger equation with the derivative nonlinearity}, Adv.
Diff. Eq. {\bf 4} (1999), 561--680.
\bibitem{Ts} Y. Tsutsumi \emph{$L^2$--solutions for nonlinear Schr\"odinger
equations and nonlinear groups}, Funkcial. Ekvac. {\bf 30} (1987),
115-125.
\end{thebibliography}
\end{document}